\documentclass[11pt]{article}
\usepackage[utf8]{inputenc}
\usepackage{amsmath,amsthm,verbatim,amssymb,amsfonts,amscd, graphicx}
\usepackage{graphics,caption}
\usepackage{hyperref}
\hypersetup{
    colorlinks=true,   	
    linkcolor=red,      
    citecolor = [rgb]{0 0.7 0},   	
    filecolor=magenta, 	
    urlcolor=blue
}
\usepackage{enumitem}
\usepackage{apptools}
\usepackage{titlesec}
\usepackage{comment}
\usepackage{color,bm}
\usepackage[baseline]{euflag}

\title{Relative entropy bounds for sampling with and without replacement}
\author{Oliver Johnson\thanks{School of Mathematics, 
	University of Bristol, Fry Building, Woodland Road, Bristol, 
	BS8 1UG, UK.
	Email: {\tt O.Johnson@bristol.ac.uk}
	} 
\and 
Lampros Gavalakis\thanks{Univ Gustave Eiffel, Univ Paris Est Creteil, CNRS, 
	LAMA UMR8050 F-77447 Marne-la-Vall{\'e}e, France. 
	Email: {\tt lampros.gavalakis@univ-eiffel.fr}. 
	L.G. has received funding from the European Union's Horizon 2020 
	research and innovation program
 	under the Marie Sklodowska-Curie grant agreement 
	No 101034255 {\euflag} and by the B{\'e}zout Labex, funded by ANR, 
	reference ANR-10-LABX-58.}
\and 
Ioannis Kontoyiannis%
\thanks{Statistical Laboratory, DPMMS,
	University of Cambridge,
	Centre for Mathematical Sciences,
        Wilberforce Road,
	Cambridge CB3 0WB, U.K.
        Email: {\tt yiannis@maths.cam.ac.uk}.
}
}
\date{\today}

\newcommand{\pr}{{\mathbb P}}
\newcommand{\ep}{{\mathbb E}}

\newcommand{\var}{{\rm Var}}

\newcommand{\AAA}{{\mathcal{A}}}
\newcommand{\wt}[1]{{\widetilde{#1}}}
\newcommand{\vc}[1]{{\boldsymbol{#1}}}

\newcommand{\supp}{{\rm supp}}
\newcommand{\typ}{{\mathcal T}}

\newcommand{\harmat}{Harremo\"{e}s and Mat\'{u}\v{s} \cite{harremoes:20}}

\newcommand{\II}{{\mathbb I}}

  \newtheorem{theorem}{Theorem}[section]
  \newtheorem{corollary}[theorem]{Corollary}
  \newtheorem{proposition}[theorem]{Proposition}
  \newtheorem{lemma}[theorem]{Lemma}
  \newtheorem{definition}[theorem]{Definition}
  
  \newtheorem{remark}[theorem]{Remark}

\begin{document}

\maketitle

\begin{abstract}
Sharp, nonasymptotic bounds are obtained for the relative entropy 
between the distributions of sampling with and without replacement 
from an urn with balls of $c\geq 2$ colors.  
Our bounds are asymptotically tight in certain regimes and,
unlike previous results, 
they depend on the number of balls of each colour in the urn.
The connection of these results with finite de Finetti-style
theorems is explored, and it is observed that a sampling bound
due to Stam (1978) combined with the convexity
of relative entropy yield a new finite de Finetti bound in relative
entropy, which achieves the optimal asymptotic convergence rate.
\end{abstract}

\section{Introduction}
\subsection{The problem}

Consider an urn containing $n$ balls, each ball having one 
of $c\geq 2$ different colours, and suppose we draw out $k\leq n$ of them. 
We will compare the difference in sampling with and without replacement. 
Write $\vc{\ell} = (\ell_1,\ell_2, \ldots, \ell_c)$ for the vector representing 
the number of balls of each colour in the urn, so that there are 
$\ell_i$ balls of colour $i$, $1\leq i\leq c$ 
and $\ell_1+\ell_2+\cdots+\ell_c = n$. 
Write $\vc{s} = (s_1,s_2, \ldots, s_c)$ for the number of balls 
of each colour drawn out, so that $s_1+s_2+\cdots+s_c = k$.

When sampling without replacement, the probability that the
colors of the $k$ balls drawn are given by $\vc{s}$ 
is given by the multivariate hypergeometric 
probability mass function (p.m.f.),
\begin{equation}
H(n,k, \vc{\ell}; \vc{s}) 
:= \frac{ \prod_{i=1}^c \binom{\ell_i}{s_i}}{\binom{n}{k}}
 = 
\frac{ \binom{k}{\vc{s}} \binom{n-k}{\vc{\ell}- \vc{s}}}{\binom{n}{\vc{\ell}}},
\label{eq:MVhyp}
\end{equation}
for all $\vc{s}$ with
$0\leq s_i\leq\ell_i$ for all $i$, and
$s_1+\cdots+s_c=k$. In~(\ref{eq:MVhyp}) and
throughout, 
we write $\binom{v}{\vc{u}} = v!/\prod_{i=1}^c u_i!$ for the 
multinomial coefficient for any 
vector $\vc{u}=(u_1,\ldots,u_c)$ with 
$u_1+\cdots+u_c = v$.

Our goal is to compare 
$H(n,k, \vc{\ell}; \vc{s})$
with the corresponding p.m.f.\ $B(n, k, \vc{\ell}; \vc{s})$
of sampling with replacement,
which is given by the multinomial distribution p.m.f.,
\begin{equation} 
\label{eq:multin}
B(n, k, \vc{\ell}; \vc{s}) 
:= \binom{k}{\vc{s}} \prod_{i=1}^c \left( \frac{\ell_i}{n} \right)^{s_i},\end{equation}
for all $\vc{s}$ with
$s_i\geq 0$ and
$s_1+\cdots+s_c=k$.

The study of the relationship between sampling 
with and without replacement
has a long history and numerous applications
in both statistics and probability;
see, e.g., the classic review~\cite{rao:66} or the 
text~\cite{thompson:12}.
Beyond the elementary observation that $B$ and $H$ have
the same mean, namely, that,
$$\sum_{\vc{s}} B(n,k,\vc{\ell}; \vc{s}) s_i 
=\sum_{\vc{s}} H(n,k,\vc{\ell}; \vc{s}) s_i = k \ell_i/n,$$
for each $i$, 
it is well known that in certain limiting regimes 
the p.m.f.s themselves
are close in a variety of senses. 
For example,
Diaconis and Freedman~\cite[Theorem~4]{diaconis-freedman:80b} 
show 
that $H$ and $B$ are close in the total variation 
distance $\| P-Q \| := \sup_A |P(A) - Q(A)|$:
\begin{equation} 
\label{eq:diaconis_s}
\|H(n,k, \vc{\ell}; \cdot)-B(n,k, \vc{\ell}; \cdot)\| \leq \frac{ck}{n}.
\end{equation}

In this paper we give bounds on the relative entropy 
(or Kullback-Leibler divergence) between $H$ and
$B$, which for brevity we denote as:
\begin{eqnarray} 
D(n,k, \vc{\ell}) & := &
D \left( H(n,k, \vc{\ell}; \cdot) \| B(n,k, \vc{\ell}; \cdot) \right) 
\nonumber \\
& = & \sum_{\vc{s}} H(n,k, \vc{\ell}; \vc{s}) \log \Big( 
\frac{ H(n,k, \vc{\ell}; \vc{s})}{B(n,k, \vc{\ell}; \vc{s})} \Big).
\label{eq:defrelent}
\end{eqnarray}
[All logarithms are natural logarithms to base $e$.]
Clearly, if $\ell_r = 0$ for some $r$,
then bounding $D(n,k;\vc{\ell})$ becomes equivalent 
to the same problem with a smaller number of colours $c$, 
so we may assume that each $\ell_r \geq 1$ for simplicity.

Interest in controlling the relative entropy derives
in part from the fact that it
naturally arises in many core statistical problems, 
e.g., as the optimal hypothesis testing error exponent
in Stein's lemma; see, e.g.,~\cite[Theorem~12.8.1]{cover:book2}. 
Further, bounding the relative entropy also provides
control of the distance between $H$ and $B$ in other senses. 
For example, Pinsker's inequality,
e.g.~\cite[Lemma~2.9.1]{tsybakov:09},
states that for any two 
p.m.f.s $P$ and $Q$,
\begin{equation} \label{eq:pinsker} 
2\| P - Q \|^2 \leq D(P \| Q),
\end{equation}
while the Bretagnolle-Huber bound~\cite{bretagnolle:78,canonne:22}
gives:
\begin{equation*} 
\| P - Q \|^2 \leq 1 - \exp\{-D(P \| Q)\}.
\end{equation*}

Our results also follow along a long line of work that has been
devoted to establishing probabilistic theorems in terms of relative 
entropy. Information-theoretic arguments often 
provide insights into the underlying reasons why the result
at hand holds. Among many others, examples include the 
well-known work of Barron~\cite{barron:clt} on the information-theoretic 
central limit theorem;
information-theoretic proofs of Poisson convergence
and Poisson approximation~\cite{konto-H-J:05}; 
compound Poisson approximation bounds~\cite{Kcompound:10};
convergence to Haar measure on compact groups~\cite{harremoes:09b}; 
connections with extreme value theory~\cite{johnson:24};
and the discrete central limit theorem~\cite{gavalakis:clt}.
We refer 
to~\cite{gavalakis-LNM:23} for a detailed review of this line of work. 

\newpage

The relative entropy in~\eqref{eq:defrelent} has been studied before.
Stam~\cite{stam:78} established the following bound,
\begin{equation} 
\label{eq:stam}
 D \left( n,k, \vc{\ell} \right) 
\leq \frac{(c-1) k(k-1)}{2(n-1)(n-k+1)},
\end{equation}
uniformly in $\vc{\ell}$,
and also provided an asymptotically matching lower bound,
\begin{equation} \label{eq:lowerbd}
D \left( n,k, \vc{\ell} \right) \geq \frac{(c-1) k(k-1)}{4(n-1)^2},
\end{equation}
indicating that, 
in the regime $k=o(n)$ and fixed $c$,
the relative entropy
$D \left( n,k, \vc{\ell} \right)$ is of order $(c-1) k^2/2n^2$.

More recently, related bounds were 
established by \harmat, who showed that (see \cite[Theorem 4.5]{harremoes:20}), 
\begin{equation}
\label{eq:harmat}
 D \left( n,k, \vc{\ell} \right) \leq (c-1) \left(  \log \left( \frac{n-1}{n-k} \right) - \frac{k}{n} + \frac{1}{n-k+1} \right),
\end{equation}
and (see \cite[Theorem 4.4]{harremoes:20}),
\begin{equation}
\label{eq:harmat2}
 D \left( n,k, \vc{\ell} \right) \geq \frac{(c-1)}{2} \left( r - 1 - \ln r  \right), \mbox{\qquad for } r = \frac{n-k+1}{n-1},
\end{equation}
both results also holding uniformly in $\vc{\ell}$.
Moreover, in the regime where $k/n \rightarrow s$ 
and $0 < \epsilon \leq \ell_r/n$ for all $r$, 
they showed that the lower bound in~\eqref{eq:harmat2} is asymptotically sharp and
the upper bound in~\eqref{eq:harmat} is sharp to within a factor of 2 by proving that:
\begin{equation} \label{eq:limit}
\lim_{n \rightarrow \infty} D \left( n,k, \vc{\ell} \right) 
= \frac{c-1}{2} \left(  - \log(1-s) -s \right) 
\approx 
(c-1) \left( \frac{s^2}{4} + \frac{s^3}{6} \right).
\end{equation}

\subsection{Main results}
\label{sec:discussion}

Despite the fact that~\eqref{eq:stam} and~\eqref{eq:harmat} are both
asymptotically optimal in the sense described above, it turns out
it is possible to obtain more accurate upper bounds 
on $D(n,k,\vc{\ell})$ if we allow them to
depend on~$\vc{\ell}$.\footnote{Indeed, as remarked
by \harmat, ``a good upper bound 
should depend on $\vc{\ell}$'' (in our notation).}
In this vein, our main result is the following;
it is proved in Section~\ref{sec:entbd}.

\begin{theorem} 
\label{thm:main} 
For any $1\leq k \leq n/2$, 
$c\geq 2$, and $\vc{\ell}=(\ell_1,\cdots,\ell_c)$
with $\ell_1+\cdots+\ell_c=n$ and $\ell_r\geq 1$
for each $r$, 
the relative entropy 
$D(n,k,\vc{\ell})$ between $H$ and $B$ satisfies:
\begin{eqnarray}
D(n,k, \vc{\ell}) & \leq &
\frac{c-1}{2} \Big( \log \Big( \frac{n}{n-k} \Big)
- \frac{k}{n-1} \Big)  
\nonumber \\
&  &  +  \frac{k(2n+1)}{12n(n-1)(n-k)} \sum_{i=1}^c \frac{n}{\ell_i}
+ \frac{1}{360} \Big( \frac{1}{(n-k)^3}  - \frac{1}{n^3} \Big) 
\sum_{i=1}^c \frac{n^3}{\ell_i^3}.\qquad 
\label{eq:simplebd}
\end{eqnarray}
\end{theorem}

While the bound of Theorem~\ref{thm:main} is simple,
straightforward, and works well for the balanced cases 
described below, it is least accurate in the case
of small $\ell_i$; see Figure~\ref{fig:smallk}.
We therefore also provide an alternative expression 
in Proposition~\ref{prop:alt}. For simplicity, we only 
give the result in the case $c=2$, though similar 
arguments will work in general. It is proved
in Appendix~\ref{sec:appsmallL}. 

\begin{proposition} \label{prop:alt}
Under the assumptions of Theorem~\ref{thm:main},
if $c=2$ and $\ell \leq n/2$, we have:
\begin{eqnarray*}
 D ( n,k, (\ell,n-\ell) ) & \leq &
\ell \Big[ \Big(1-\frac{k}{n}\Big) \log\Big(1-\frac{k}{n}\Big) 
	+ \frac{k}{n} - \frac{k}{2 n(n-1)}  \Big]   
	\\
& & + \frac{k \ell}{(n-1) (n-\ell)(n-k)}
	\\
& &  - \frac{k(k-1)}{2n(n-1)}  \ell(\ell-1) \log 
	\Big( \frac{\ell}{\ell-1} \Big) 
	\\
& & - \frac{k(k-1)(k-2)}{6n(n-1)(n-2)}  \ell(\ell-1)(\ell-2)  
	\log \Big( \frac{(\ell-1)^2}{\ell(\ell-2)} \Big).  
\end{eqnarray*}
\end{proposition}

Figure~\ref{fig:smallk} shows a comparison of the earlier uniform
bounds~\eqref{eq:stam},~\eqref{eq:lowerbd} and~\eqref{eq:harmat},
with the new bounds in Theorem~\ref{thm:main} 
and Proposition~\ref{prop:alt},
for the case $n=100$, $k=30$, and $\vc{\ell}=(\ell,1-\ell)$,
for $1\leq\ell\leq 50$. It is seen that the combination
of the two new upper bounds outperforms the earlier ones
in the entire range of values considered.

\begin{figure}[ht!]
\centerline{\includegraphics[width=9cm]{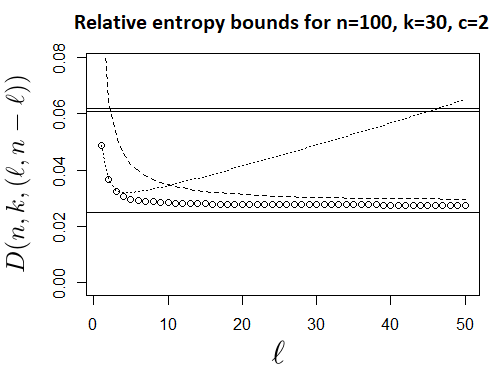}}
\caption{Comparison of the true values of $D(100,30,(\ell,100-\ell))$,
plotted as circles,
with the uniform upper bounds in~\eqref{eq:stam},~\eqref{eq:harmat} 
and the uniform lower bound~\eqref{eq:harmat2}, 
all plotted as straight lines, and also with
the new bounds in Theorem~\ref{thm:main} 
(dashed line)
and Proposition~\ref{prop:alt} (dotted line).
The combination of the two new bounds outperforms those 
of Stam~\cite{stam:78} and of \harmat\ for the whole
range of the values of $\ell$ in this case.}
\label{fig:smallk} 
\end{figure}

\noindent
{\bf Remarks. }
\begin{enumerate}
\item
The dependence of the bound~(\ref{eq:simplebd})
in Theorem~\ref{thm:main}
on $\vc{\ell}$ is only via the quantities:
\begin{equation} 
\label{eq:invell}
\Sigma_1(n,c,\vc{\ell}):=\sum_{r=1}^c \frac{n}{\ell_r}
\quad\mbox{and}\quad
\Sigma_2(n,c,\vc{\ell}):=\sum_{r=1}^c \frac{n^3}{\ell_r^3}.
\end{equation}
For $\Sigma_1$, 
in the `balanced' case where each $\ell_r = n/c$, 
we have $\Sigma_1(n,c,\vc{\ell})= c^2$, and using Jensen's inequality it is
easy to see that we always have $\Sigma_1(n,c,\vc{\ell})\geq c^2$.
In fact, the Schur-Ostrowski criterion shows that~(\ref{eq:invell})
is Schur convex, and hence respects the majorization order, 
being largest in `unbalanced' cases.
For example, if
$\ell_1 = \cdots = \ell_{c-1} = 1$ and $\ell_c = n-c+1$,
then $\Sigma_1(n,c,\vc{\ell})=(c-1) n + 1/(n-c+1) \sim n(c-1)\to\infty$ 
as $n \rightarrow \infty$. 
On the other hand, in the regime considered by~\harmat\ where all 
$\ell_r \geq \epsilon n$ for some $0 < \epsilon \leq 1/c$,
we have $\Sigma_1(n,c,\vc{\ell}) \leq c/\epsilon$, 
which is bounded in~$n$.

Similar comments apply to $\Sigma_2(n,c,\vc{\ell})$:
It is bounded in balanced cases, including 
the scenario of \harmat, but can grow like $n^3(c-1)$ in 
the very unbalanced case.


\item
The first term in the bound~(\ref{eq:simplebd}) of Theorem~\ref{thm:main},
$$\frac{c-1}{2} \Big( \log \Big(\frac{n}{n-k} \Big) 
- \frac{k}{n-1} \Big),$$ 
matches the asymptotic expression in~(\ref{eq:limit}), so 
Theorem~\ref{thm:main} may be
regarded as a nonasymptotic version version of~(\ref{eq:limit}).
In particular,~\eqref{eq:limit} is only valid in 
the asymptotic case where all $\ell_r \geq \epsilon n$, 
when, as discussed above,
$\Sigma_1(n,c,\vc{\ell})$ is bounded. Therefore, in this balanced 
case Theorem~\ref{thm:main} recovers the asymptotic result 
of~\eqref{eq:limit}. 

\item
Any bound that holds uniformly in $\vc{\ell}$ must hold in particular 
for the binary $c=2$ case with $\ell_1 = 1, \ell_2=n-1$. 
Then $H$ is Bernoulli with 
parameter $k/n$ and $B$ is binomial with parameters $k$ and~$1/n$, 
and direct calculation 
gives the exact expression:
\begin{equation*}
D \left( n,k, (1,n-1) \right)
 =  \left(1- \frac{k}{n} \right) 
\log \left(1- \frac{k}{n} \right) - 
k \left(1 - \frac{1}{n} \right) \log \left(1 - \frac{1}{n} \right). 
\end{equation*}
Taking $n\to\infty$ while $k/n \rightarrow s$ 
for some fixed $s$, in this unbalanced case where $\ell_1=1$ and $\ell_2=n-1$,
we obtain:
\begin{equation} 
\label{eq:limit2}
\lim_{n \rightarrow \infty} D \left( n,ns, (1,n-1) \right) = s + (1-s) \log(1-s) \approx \frac{s^2}{2} + \frac{s^3}{6}.
\end{equation}
Compared with the analogous limiting expression~\eqref{eq:limit}
in the balanced case where $\ell_1,\ell_2$ are both bounded
below by $\epsilon n$, 
this suggests that the asymptotic behaviour of $D(n,k,\vc{\ell})$ 
in the regime $k/n\to s$
may vary by a factor of~2 over different values of $\ell$.

\item
Comparing the earlier asymptotic expression~\eqref{eq:limit} 
with~\eqref{eq:limit2}, reveals some interesting behaviour 
of $D(n,k, (\ell,n-\ell))$. 
There is a unique value $s^* \approx 0.8834$ such that,
for $s < s^*$, corresponding to small $k$, 
the limiting expression in~\eqref{eq:limit2} is larger 
than that in~\eqref{eq:limit}, 
whereas, for $s > s^*$, the expression in~\eqref{eq:limit} is larger.
\end{enumerate}

\subsection{Finite de Finetti theorems} 
\label{sec:impdefi}

Diaconis~\cite{diaconis:77}
and Diaconis and Freedman~\cite{diaconis-freedman:80b}
revealed an interesting and intimate connection between
the problem of comparing sampling with and without
replacement, and finite versions of de Finetti's theorem.
Recall that de Finetti's theorem says that the distribution
of a infinite collection of exchangeable random variables 
can be represented as a mixture over independent and
identically distributed (i.i.d.) sequences;
see, e.g.,~\cite{kirsch:19} 
for an elementary proof in the binary case.
Although such a representation does not always exist
for {\em finite} exchangeable sequences, 
approximate versions are possible:
If $(X_1,\ldots,X_n)$ are exchangeable, then the distribution
of $(X_1,\ldots,X_k)$ is close to a mixture of independent
sequences as long as $k$ is relatively small compared to $n$.
Indeed, 
Diaconis and Freedman~\cite{diaconis-freedman:80b} proved 
such a finite de Finetti theorem using
the sampling bound~\eqref{eq:diaconis_s}. 

Let $\AAA= \{ a_1, \ldots, a_c \}$ be an alphabet of size $c$,
and suppose that the random variables
$\vc{X}_n = (X_1, \ldots, X_n)$, with values $X_i \in \AAA$,
are exchangeable, that is, the distribution of $\vc{X}_n$
is the same as that
of $\left(X_{\pi(1)}, \ldots, X_{\pi(n)} \right)$ for any 
permutation $\pi$ on $\{1,\ldots,n\}$.
Given a sequence $\vc{y} \in \AAA^m$ of length $m$,
its {\em type} $\typ_m(\vc{y})$~\cite{csiszar:98} 
is the vector of its empirical frequencies: The 
$r$th component of $\typ_m(\vc{y})$ is:
$$ \frac{1}{m} \sum_{i=1}^m \II(y_i = a_r),
\quad 1\leq r\leq c.$$
The sequence $\vc{X}_n$ induces a measure $\mu$ on 
p.m.f.s $\vc{p}$ on $\AAA$, via:
\begin{equation} \label{eq:empirical} \mu(\vc{p}) = 
\pr \left(\typ_n(\vc{X}_n) = \vc{p} \right).\end{equation}
This is the law of the empirical measure induced
by $\vc{X}_n$ on $\AAA$.

For each $1\leq k\leq n$, let  $P_k$ denote the
joint p.m.f.\ of $\vc{X}_k = (X_1, \ldots, X_k)$,
and write,
\begin{equation}
M_{k,\mu}(\vc{z}) = 
\int \vc{p}^k(\vc{z}) d\mu(\vc{p}),
\quad \vc{z}\in\AAA^k,
\label{eq:mixture}
\end{equation}
for the mixture of the i.i.d.\ distributions 
$\vc{p}^k$ with respect to the mixing measure $\mu$. 
A key step in the connection between sampling bounds and 
finite de Finetti theorems is the simple observation that,
given that its type $\typ_k(\vc{X}_k)=\vc{q}$,
the exchangeable sequence $\vc{X}_k$ is uniformly distributed on the 
set of sequences with type~$\vc{q}$. 
In other words, for any $\vc{x}\in\AAA^k$,
\begin{equation}
P_k(\vc{x}) = 
\frac{1}{ \binom{k}{\vc{s}}}
\sum_{\vc{p}} H(n,k, n \vc{p}; \vc{s})  \mu(\vc{p}),
\label{eq:Pkrep}
\end{equation}
where $\vc{s}=k\typ_k(\vc{x})$.
This elementary and rather obvious (by symmetry)
observation already appears implicitly in the 
literature, e.g., in~\cite{diaconis:77,gavalakis-LNM:23}.

An analogous representation can be easily seen 
to hold for $M_{k,\mu}$. Since the probability
of an i.i.d.\ sequence only depends on its type,
for any $\vc{x}\in\AAA^k$ we have, 
with $\vc{s}=k\typ_k(\vc{x})$:
\begin{equation}
M_{k,\mu}(\vc{x})  =  \frac{1}{ \binom{k}{\vc{s}}}
\sum_{\vc{p}} B(n,k, n \vc{p}; \vc{s}) \mu(\vc{p}).
\label{eq:Mkrep}
\end{equation}

The following simple proposition clarifies
the connection between finite de Finetti theorems
and sampling bounds. Its proof is a simple 
consequence of~(\ref{eq:Pkrep}) combined
with~(\ref{eq:Mkrep}) and the 
log-sum inequality~\cite{cover:book2}.

\begin{proposition} \label{prop:mixing}
Suppose $\vc{X}_n=(X_1,\ldots,X_n)$ is a finite
exchangeable sequence of random variables with
values in $\AAA$, and
let the $\mu$ denote the law of the empirical measure 
defined by~\eqref{eq:empirical}.
Then for any $k\leq n$, we have,
\begin{equation*} 
D( P_k \| M_{k,\mu}) 
\leq \sum_{\vc{p}} \mu(\vc{p}) D \left( n,k, n \vc{p} \right)
\leq \max_{\frac{\vc{\ell}}{n} \in \supp(\mu)}
D \left( n,k, \vc{\ell} \right),
\end{equation*}
where $M_{k,\mu}$ is the mixture defined in~{\em (\ref{eq:mixture})}
and the maximum is taken over all p.m.f.s of
the form $\vc{\ell}/n$ in the support of $\mu$.
\end{proposition}

\begin{proof}
Fix $\vc{s}$ with $\sum_{i=1}^c s_i = k$. The log-sum inequality gives 
that for any $\vc{x}$ of type $\typ_k(\vc{x}) = \vc{s}/k$:
\begin{align*}
P_k( \vc{x}) \log \left( \frac{ P_k( \vc{x})}{M_{k,\mu}(\vc{x})} \right)
& = 
	\binom{k}{\vc{s}}^{-1}
	\left( \sum_{\vc{p}} H(n,k, n \vc{p}; \vc{s}) \mu(\vc{p}) \right) 
	\log \left( \frac{ \sum_{\vc{p}} H(n,k, n \vc{p}; \vc{s}) \mu(\vc{p}) }
	{\sum_{\vc{p}} B(n,k, n \vc{p}; \vc{s}) \mu(\vc{p})} \right) 
	\\
& \leq  
	\binom{k}{\vc{s}}^{-1} \sum_{\vc{p}} \mu(\vc{p}) 
	H(n,k, n \vc{p}; \vc{s})  
	\log \left( \frac{ H(n,k, n \vc{p}; \vc{s}) }
	{B(n,k, n \vc{p}; \vc{s})} \right). 
\end{align*}
Hence, summing over the $\binom{k}{\vc{s}}$ vectors $\vc{x}$ of type 
$\typ_k(\vc{x}) = \vc{s}/k$ we obtain,
$$ \sum_{\vc{x}: \typ(\vc{x}) = \vc{s}/k} P_k( \vc{x}) \log \left( \frac{ P_k( \vc{x})}{M_{k,\mu}(\vc{x})} \right) 
 \leq  \sum_{\vc{p}} \mu(\vc{p}) H(n,k, n \vc{p}; \vc{s})  \log \left( \frac{ H(n,k, n \vc{p}; \vc{s}) }{B(n,k, n \vc{p}; \vc{s})} \right). $$
Finally, summing over $\vc{s}$ yields,
$$ D( P_k \| M_{k,\mu}) \leq \sum_{\vc{p}} \mu(\vc{p}) D \left( n,k, n \vc{p} \right),$$
and each of the relative entropy terms can be bounded by the 
maximum value.
\end{proof}

Combining Stam's bound~\eqref{eq:stam} and Proposition~\ref{prop:mixing} 
immediately gives:

\begin{corollary}[Sharp finite de Finetti] 
\label{cor:defi}
Under the assumptions of Proposition~\ref{prop:mixing},
for $1\leq k\leq n$:
\begin{equation} 
\label{dFDbound}
D( P_k \| M_{k,\mu}) \leq 
  \frac{(c-1) k(k-1)}{2(n-1)(n-k+1)}.
\end{equation}
\end{corollary}

Pinsker's inequality~\eqref{eq:pinsker} and~\eqref{dFDbound} 
gives,
\begin{equation*}
 \| P_k - M_{k,\mu} \| \leq \sqrt{\frac{(c-1) k(k-1)}{2(n-1)(n-k+1)}}.
\end{equation*}
In the $k=o(n)$ regime this is in fact optimal,
in that it is of the same order as 
Diaconis and Freedman's~\cite{diaconis-freedman:80b}
bound,
\begin{equation} \label{eq:diaconis}
    \| P_k - M_{k,\mu} \| \leq \frac{ c k}{n},
\end{equation}
which they show is of optimal order in $k$ and $n$
when $k=o(n)$.

For certain applications, for example
to approximation schemes for minimization 
of specific polynomials~\cite{berta:22}, 
the dependence on the alphabet size is also of interest;
see also~\cite{gavalakis-arxiv:23} and the references therein. 
The lower bound~\eqref{eq:lowerbd} shows that the linear 
dependence on the alphabet size $c$ is optimal for the sampling problem, 
in the sense that any upper bound that holds for any $c$, $k$, and $n$, 
must have at least linear dependence on $c$. 
However, we do not know 
whether the linear dependence on the alphabet size $c$ is optimal 
for the de Finetti problem under optimal rates in $k,n$. 

Information-theoretic proofs of finite de Finetti theorems 
have been lately developed 
in~\cite{gavalakis:21,gavalakis-LNM:23,gavalakis-arxiv:23}. 
In~\cite{gavalakis:21}, the bound,
\begin{equation} \label{gkfirsteq}
D(P_{k}\|M_{k,\mu}) \leq \frac{5k^2\log n}{n-k},
\end{equation}
was obtained for binary sequences,
and in~\cite{gavalakis-LNM:23},
the weaker bound,
\begin{equation} \label{gksecondeq}
D(P_{k}\| M_{k,\mu})=
O\left(\Big(\frac{k}{\sqrt{n}}\Big)^{1/2}\log{\frac{n}{k}}\right),
\end{equation}
was established for finite-valued random variables.
Finally the sharper bound,
\begin{equation} \label{gkbeq}
D(P_{k}\|M_{k,\mu}) \leq \frac{k(k-1)}{2(n-k-1)}\log c,
\end{equation}
was derived in~\cite{gavalakis-arxiv:23}, where 
random variables with values in abstract spaces were also considered. 
Although the derivations of~\eqref{gkfirsteq}--\eqref{gkbeq} 
are interesting in that they employ purely information-theoretic
ideas and techniques, the actual bounds
are of strictly weaker rate than the sharp $O(k^2/n^2)$ rate 
we obtained in Corollary~\ref{cor:defi} via sampling bounds.

\medskip

\noindent
{\bf A question on monotonicity.} 
A significant development in the area of information-theoretic
proofs of probabilistic limit theorems was 
in 2004, when it was shown~\cite{artstein:04}
that the convergence in relative entropy established
by Barron in 1986~\cite{barron:clt} was in fact monotone.
In the present setting we observe that,
while the total variation 
bound~\eqref{eq:diaconis} of Diaconis 
and Freedman~\cite{diaconis-freedman:80b}  
$\| P_k - M_{k,\mu} \|\leq \frac{2 c k}{n}$ is nonincreasing in $n$, 
we do not know whether the total variation distance 
is itself monotone.

Writing $D( P_r \| M_{r,\mu}) = 
D((X_1,\ldots,X_r)\|(\tilde{X}_1,\ldots,\tilde{X}_r))$, 
where the vector $(\tilde{X}_1,\ldots,\tilde{X}_r)$ is distributed 
according to $M_{r,\mu}$, a direct application of the data processing 
inequality for the relative entropy~\cite{cover:book2}
gives that, for any integer $r, s \geq 0$:
$$ D( P_r \| M_{r,\mu}) \leq D( P_{r+s} \| M_{r+s,\mu}).$$
In other words, $D( P_k \| M_{k,\mu})$ is nondecreasing in $k$. Since the total variation is an $f$-divergence and therefore satisfies the data processing inequality, the same argument shows that
$\|P_k - M_{k,\mu}\|$
is also nondecreasing in $k$. In view of the monotonicity in the 
information-theoretic central limit theorem and the convexity properties 
of the relative entropy it is tempting to conjecture 
that $D( P_k \| M_{k,\mu})$ is 
{\em also nonincreasing in $n$}. However, this relative entropy depends 
on $n$ through the choice of the mixing measure $\mu$, which makes 
the problem significantly harder. 
We therefore ask:

{\em Let $\vc{X}_{n+1} = (X_1,\ldots,X_{n+1})$ be a finite exchangeable
random sequence and let the mixing measures $\mu_n, \mu_{n+1}$ be the laws 
of the empirical measures induced by $\vc{X}_n$ and $\vc{X}_{n+1}$,
respectively. For a fixed $k\leq n$, is 
$D( P_k \| M_{k,\mu_n}) \leq D( P_k \| M_{k,\mu_{n+1}})$? 
Are there possibly different measures
${\nu}_n$ such that $D( P_k \| M_{k,\nu_n})$ 
is vanishing for $k = o(n)$ and nonincreasing in $n$?}

\section{Upper bounds on relative entropy} \label{sec:entbd}

The proof of Theorem~\ref{thm:main} in this section
will be based on the 
decomposition of $D(n,k,\vc{\ell})$ as a sum 
of expectations of terms involving the quantity $U(a,b)$ 
introduced in Definition~\ref{def:Udef}. We tightly approximate 
these terms $U$ within a small additive error 
(see Proposition~\ref{prop:Ubd}), and control the 
expectations of the resulting terms using
Lemmas~\ref{lem:sumform1} and~\ref{lem:sumform3}.

\subsection{Hypergeometric properties}

We first briefly mention some standard properties of hypergeometric 
distributions that will help our analysis.

Notice that~\eqref{eq:MVhyp} and~\eqref{eq:multin} are both invariant 
under permutation of colour labels. That is,
 $H(n, k, \vc{\ell}; \vc{s}) = H(n, k, \wt{\vc{\ell}}; \wt{\vc{s}})$ 
and $B(n, k, \vc{\ell}; \vc{s}) = B(n, k, \wt{\vc{\ell}}; \wt{\vc{s}})$,
whenever $\wt{\ell}_i = \ell_{\pi(i)}$ and $\wt{s}_i = s_{\pi(i)}$ for
the same permutation $\pi$. Hence $D(n, k, \vc{\ell})$ 
from~\eqref{eq:defrelent} is itself permutation invariant, 
\begin{equation} 
\label{eq:swap}  
D(n,k, \vc{\ell}) = D(n,k, \wt{\vc{\ell}}),
\end{equation}
so we may assume, without loss of 
generality, that $\ell_1 \leq \ell_2 \leq \cdots \leq \ell_c$.
Similarly, if we swap the roles of balls that are `drawn' and `not drawn', 
then direct calculation using the second representation in~\eqref{eq:MVhyp} 
shows that,
\begin{equation} \label{eq:swap2}
H(n,n-k,\vc{\ell}; \vc{\ell}- \vc{s}) = \frac{\binom{n-k}{\vc{\ell} - \vc{s}} \binom{k}{\vc{s}}}{\binom{n}{\vc{\ell}}}
= H(n,k, \vc{\ell}; \vc{s}).\end{equation} 

Let $\vc{S}=(S_1,\ldots,S_c)\sim H(n,k,\vc{\ell};\cdot)$.
Then each $S_r$ has a multinomial distribution with $c=2$, 
with p.m.f.\ given by 
$\pr(S_r = s_r)=
H(n,k,\ell_r;s_r):= H(n,k, (\ell_r,n-\ell_r); (s_r,k-s_r))$.

We use standard notation and write $(x)_r =
x(x-1) \ldots (x-r+1) = x!/(x-r)!$ for the falling factorial, and note that on marginalizing we can apply a standard result:

\begin{lemma} \label{lem:fallmom} 
For any $i$ and $r$, the factorial moments of the
hypergeometric $H(n,k, \vc{\ell}; \cdot)$ satisfy:
\begin{equation*}
\ep[(S_i)_r] =
\sum_{\vc{s}} H(n,k, \vc{\ell}; \vc{s}) (s_i)_r  = 
\frac{(\ell_i)_r (k)_r}{(n)_r}. 
\end{equation*}
\end{lemma}

Hence, as mentioned in the Introduction,
the mean $\ep(S_i) = k \ell_i/n$. In general, we write 
$M_r(S_i) = \ep[(S_i - k \ell_i/n)^r]$ for the $r$th centered moment 
of $S_i$ and note that, by expressing $(s_i - k \ell_i/n)^r$ as a linear 
combination of factorial moments of $s_i$, we obtain that,
\begin{eqnarray}
M_2(S_i) & = & \var(S_i) = \frac{k(n-k) \ell_i(n-\ell_i)}{n^2 (n-1)}, \label{eq:m2val} \\
M_3(S_i) & = & \frac{k \ell_i (n - k) (n - 2 k) (n-\ell_i) (n-  2\ell_i)}{n^3(n-1)(n-2)}. \label{eq:m3val}
\end{eqnarray}

\subsection{Proof of Theorem~\ref{thm:main}}

A key role in the proof will be played by the analysis of the following 
quantity, defined in terms of the gamma function $\Gamma$:

\begin{definition} \label{def:Udef}
For $0 \leq b \leq a$, define the function:
\begin{equation} \label{eq:Udef}
U(a, b) := \log \left( \frac{ a^b \Gamma(a-b+1)}{\Gamma(a+1)} \right).
\end{equation}
\end{definition}

Using the standard fact that $\Gamma(n+1) = n!$, when $a$ and $b$ are 
both integers we can write,
\begin{equation*}
U(a, b) := \log \left( \frac{ a^b (a-b)!}{a!}
\right) = \log \left( \frac{a^b}{a(a-1) \cdots (a-b+1)}
\right) \geq 0,
\end{equation*}
and note that $U(a,0) = U(a,1) \equiv 0$ for all $a$.
Using Stirling's approximation for the factorials suggests that 
$U(a,b) \sim (a-b+1/2) \log( (a-b)/a) + b$, but we can make this 
precise using results of Alzer~\cite{alzer:97}, for example.
Proposition~\ref{prop:Ubd} is proved in
Appendix~\ref{sec:newproofs}. 

\begin{proposition} \label{prop:Ubd}
Writing,
\begin{align*}
A(a,b) & := (a-b+1/2) \log \left( \frac{a-b}{a} \right) + b + \frac{1}{12(a-b)} - \frac{1}{12 a}, \\
\varepsilon(a,b) & :=  \frac{1}{360} \left( \frac{1}{(a-b)^3} - \frac{1}{a^3} \right) \geq 0,
\end{align*}
for any $0 \leq b \leq a-1$ we have the bounds:
\begin{equation} \label{eq:Ubd}
 A(a,b) - \varepsilon(a,b) \leq       U(a,b) \leq  A(a,b).
 \end{equation}
\end{proposition}

\begin{proof}[Proof of Theorem~\ref{thm:main}]
We can give an expression for the relative 
entropy $D(n,k, \vc{\ell})$ by using the second form 
of~\eqref{eq:MVhyp} and the fact that $\sum_{i=1}^c s_i =k$, to obtain,
\begin{eqnarray*}
 \log \left( \frac{ H(n,k, \vc{\ell}; \vc{s})}{B(n,k, \vc{\ell}; \vc{s})} \right)
&  = &
 \log \left( \frac{\binom{n-k}{\vc{\ell} - \vc{s}}}{\binom{n}{\vc{\ell}}}
\prod_{i=1}^c  \frac{1}{(\ell_i/n)^{s_i}} \right) \nonumber \\
& = & \log \left( \frac{ n^k (n-k)!}{n!} \right) + \sum_{i=1}^c \log \left( \frac{ \ell_i!}{\ell_i^{s_i} (\ell_i - s_i)! } \right)   \nonumber \\
& = & U(n,k) - \sum_{i=1}^c U(\ell_i, s_i),
\end{eqnarray*}
so that we can express,
\begin{equation}
D \left( n,k, \vc{\ell} \right) 
 =  U(n,k) -\sum_{i=1}^c \ep U(\ell_i, S_i),
\label{eq:relentsumform}
\end{equation}
where $S_i$ has a hypergeometric $H(n,k, \ell_i;\cdot)$ 
distribution.\footnote{Note 
that since, $S_i \leq \min(k, \ell_i)$, in the case $k=1$ all of the values 
involved in this calculation are of the form $U(a,0)$ or $U(a,1)$, 
so~\eqref{eq:relentsumform} is identically zero as we would expect -- 
we only sample one item, so it does not matter whether that 
is with or without replacement.}
Now, we can approximate $\ep U(\ell_i, S_i)$ by  
$U(\ell_i, \ep S_i) = U(\ell_i, \ell_i k/n)$, to rewrite:
\begin{equation}
D \left( n,k, \vc{\ell} \right) 
 =  \Big( U(n,k) - \sum_{i=1}^c U(\ell_i, \ell_i k/n) \Big) 
+ \sum_{i=1}^c  \left( - \ep U(\ell_i, S_i) +
 U(\ell_i, \ell_i k/n)   \right).
\label{eq:relentsumforma}
\end{equation}
We bound the two parts of~\eqref{eq:relentsumforma} 
in Lemmas~\ref{lem:sumform1} and~\ref{lem:sumform3}, respectively. 
The result follows on combining these two expressions, 
noting that the result also includes the negative term,
\begin{equation*} 
\frac{k}{12 n(n-k)} - \frac{ c k }{4(n-k)(n-1)} =
- \frac{k (3 c n - n +1)}{12 n(n-1)(n-k)} 
\leq - \frac{k (3 c-1)}{12 (n-1)(n-k)} <0,
\end{equation*}
which may be ignored.
\end{proof}

Lemmas~\ref{lem:sumform1} and~\ref{lem:sumform3}
are proved in Appendix~\ref{sec:newproofs2}. 

\begin{lemma} \label{lem:sumform1}
The first term of~\eqref{eq:relentsumforma}, 
$U(n,k) - \sum_{i=1}^c U(\ell_i, \ell_i k/n)$,
is bounded above by:
\begin{eqnarray*}
\frac{c-1}{2} \log \Big( \frac{n}{n-k} \Big) 
+ \frac{k}{12 n(n-k)} \Big( 1 - \sum_{i=1}^c \frac{n}{\ell_i} \Big) 
+ \frac{1}{360} \Big( \frac{n^3}{(n-k)^3}  - 1 \Big) \sum_{i=1}^c \frac{1}{\ell_i^3}.
\end{eqnarray*}
\end{lemma}

\begin{lemma} \label{lem:sumform3}
The second term in~\eqref{eq:relentsumforma} is bounded
above as:
$$\sum_{i=1}^c  \left(- \ep U(\ell_i, S_i) +
 U(\ell_i, \ell_i k/n)   \right) 
\leq  - \frac{ k (c-1)}{2(n-1)} + \frac{k}{4(n-k)(n-1)} \sum_{i=1}^c \left( \frac{n}{\ell_i} -1 \right).
$$
\end{lemma}

\newpage

\appendix
\section{Appendix: Proofs} 
\subsection{Proof of Proposition~\ref{prop:Ubd}}
\label{sec:newproofs}

The key to the proof is to work with the logarithmic derivative of the 
gamma function $\psi(x) := \frac{d}{d x} \log \Gamma(x)$, $x>0$. 
For completeness, we state and prove the following standard bound on $\psi$:

\begin{lemma} \label{lem:psibd} For any $y > 0$:
\begin{equation*}
0 \leq \psi(y) - \log y + \frac{1}{2y} + \frac{1}{12y^2}
\leq \frac{1}{120 y^4}.
\end{equation*}
\end{lemma}
\begin{proof}
Note first that, by~\cite[Eq.~(2.2)]{alzer:97} for example, 
$$-\frac{1}{y} \leq \psi(y) - \log(y) \leq -\frac{1}{2y},$$
meaning that $\psi(y) - \log(y) \rightarrow 0$ as $y \rightarrow \infty$, so we can write:
\begin{equation} \label{eq:psibds}
\psi(y) - \log(y) = \int_y^\infty \left( - \psi'(x) + \frac{1}{x} \right) dx.\end{equation}
Further, taking $k=1$ in~\cite[Theorem~9]{alzer:97}, we 
can deduce that for all $x > 0$:
\begin{equation} \label{eq:psibds2}
 - \frac{1}{2x^2} - \frac{1}{6x^3}   \leq
- \psi'(x) + \frac{1}{x} \leq - \frac{1}{2x^2} - \frac{1}{6x^3} + \frac{1}{30x^5}.
\end{equation}
Substituting~\eqref{eq:psibds2} in~\eqref{eq:psibds} and integrating 
yields the claimed result.
\end{proof}

\begin{proof}[Proof of Proposition~\ref{prop:Ubd}]
Combining~\eqref{eq:Udef} with the standard fact that
$ \Gamma(n+1) = n \Gamma(n)$ we can write,
\begin{eqnarray*}
U(a, b) & = & b \log a + \log \Gamma(a-b) + \log(a-b) - \log \Gamma(a) - \log a \\
& = & (b-1) \log a + \log(a-b) + \int_{a-b}^a -\psi(x) dx.
\end{eqnarray*}
Now, we can provide an upper bound on this using the result $-\psi(x) \leq -\log x + 1/(2x) + 1/(12 x^2)$ from Lemma~\ref{lem:psibd} to deduce that:
\begin{eqnarray*}
U(a, b) & \leq & 
 (b-1) \log a + \log(a-b) + \left[ -(x-1/2) \log x
+ x - \frac{1}{12 x} \right]^a_{a-b} \\
& = & (a-b+1/2) \log \left( \frac{a-b}{a} \right)
+ b + \frac{1}{12(a-b)} - \frac{1}{12 a}.
\end{eqnarray*}
Similarly, since $-\psi(x)  \geq -\log x + 1/(2x) + 1/(12 x^2) - 1/(120 x^4)$ we can deduce that
\begin{eqnarray*}
 U(a,b) & \geq & A(a,b) - \int_{a-b}^a \frac{1}{120 x^4} dx
 = A(a,b) - \frac{1}{360} \left( \frac{1}{(a-b)^3} - \frac{1}{a^3} \right),
 \end{eqnarray*}
 and the result follows.
\end{proof}

\subsection{Proofs of Lemmas~\ref{lem:sumform1} and~\ref{lem:sumform2}}
\label{sec:newproofs2}

\begin{proof}[Proof of Lemma~\ref{lem:sumform1}] 
Using the fact that $\sum_{i=1}^c \ell_i = n$,
we first simplify the corresponding approximation 
terms from Proposition~\ref{prop:Ubd} into the form,
\begin{eqnarray}
\lefteqn{ A(n,k)  - \sum_{i=1}^c A(\ell_i, \ell_i k/n) } \nonumber \\
& = & \left( n - k + 1/2 - 
\left( \sum_{i=1}^c \ell_i(1-k/n) \right) - c/2 \right)
\log \left( \frac{n-k}{n} \right) + k - \frac{k}{n} \sum_{i=1}^c \ell_i \nonumber \\
& & + \frac{k}{12 n(n-k)} \left( 1 - \sum_{i=1}^c \frac{n}{\ell_i} \right) \nonumber \\
& = & \frac{c-1}{2} \log \left( \frac{n}{n-k} \right) 
+ \frac{k}{12 n(n-k)} \left( 1 - \sum_{i=1}^c \frac{n}{\ell_i} \right).
 \label{eq:simp}
\end{eqnarray}
Then we can use Proposition~\ref{prop:Ubd} to obtain that,
\begin{eqnarray*}
U(n,k) - \sum_{i=1}^c U(\ell_i, \ell_i k/n)
& \leq & A(n,k)  - \sum_{i=1}^c A(\ell_i, \ell_i k/n) + \sum_{i=1}^c \varepsilon(\ell_i, \ell_i k/n)  \\
& = &  A(n,k)  - \sum_{i=1}^c A(\ell_i, \ell_i k/n) +
 \frac{1}{360} \left( \frac{n^3}{(n-k)^3}  - 1 \right) \sum_{i=1}^c \frac{1}{\ell_i^3}, 
\end{eqnarray*}
and the upper bound follows using~\eqref{eq:simp}.
\end{proof}

In order to prove Lemma~\ref{lem:sumform3} we first provide 
a Taylor series based approximation for the 
summands in~(\ref{eq:relentsumforma}):

\begin{lemma} \label{lem:sumform2}
Each term in the second sum in~\eqref{eq:relentsumforma} satisfies:
\begin{eqnarray} \lefteqn{ - \ep U(\ell_i, S_i) + U(\ell_i, \ell_i k/n) } 
\nonumber \\
& \leq & - \psi'(\ell_i(1-k/n)+1) \frac{M_2(S_i)}{2}
+ \psi''(\ell_i(1-k/n)+1) \frac{M_3(S_i)}{6}. \label{eq:tosortout}
\end{eqnarray}
\end{lemma}
\begin{proof}
We can decompose the summand in~\eqref{eq:relentsumforma} in the form
\begin{eqnarray}  
-\ep U(\ell_i, S_i) + U(\ell_i, \ell_i k/n) 
 & = & - \ep \left( S_i \log \ell_i + \log \Gamma(\ell_i-S_i+1)  - \log \Gamma(\ell+1) \right) \nonumber \\
& &  + \left( \frac{\ell_i k}{n} \log \ell_i + \log \Gamma(\ell_i(1-k/n)+1)  - \log \Gamma(\ell+1) \right)
\nonumber \\
& = & -\log \Gamma(\ell_i-S_i+1) + \log \Gamma(\ell_i(1-k/n)+1) \label{eq:gammadiff}
 \end{eqnarray}
Recalling the definition  $\psi(x) =: \frac{d}{dx} \log \Gamma(x)$, 
we have the Taylor expansion,
\begin{eqnarray*}
\lefteqn{-\log \Gamma(\ell-s+1) + \log \Gamma(\ell-\mu+1)} \\
& = & \psi(\ell-\mu+1) (s-\mu)
- \psi'(\ell-\mu+1) \frac{(s-\mu)^2}{2}
+ \psi''(\ell-\mu+1) \frac{(s-\mu)^3}{6} 
- \psi'''(\xi) \frac{(s-\mu)^4}{24} \\
& \leq & \psi(\ell-\mu+1) (s-\mu)
- \psi'(\ell-\mu+1) \frac{(s-\mu)^2}{2}
+ \psi''(\ell-\mu+1) \frac{(s-\mu)^3}{6},
\end{eqnarray*}
where we used the fact,
e.g., by~\cite[Theorem~9]{alzer:97},
that $\psi'''(x) \geq 0$. 
The result follows by 
substituting this in~\eqref{eq:gammadiff} and taking expectations.
\end{proof}

\begin{proof}[Proof of Lemma~\ref{lem:sumform3}]
We will bound  the two terms from Lemma~\ref{lem:sumform2}.
Using~\cite[Theorem~9]{alzer:97} we can deduce that, for 
any $x >0$,
$$
- \psi'(x+1) \leq - \frac{1}{x+1} - \frac{1}{2(x+1)^2}
\leq - \frac{1}{x} + \frac{1}{2x^2}.
$$
Hence, recalling the value of $M_2(S_i)$ from~\eqref{eq:m2val},
we know that the first term of~\eqref{eq:tosortout} contributes 
an upper bound of,
\begin{eqnarray}
\lefteqn{ - \sum_{i=1}^c \psi'(\ell_i(1-k/n)+1) \frac{M_2(S_i)}{2}} \nonumber \\
& \leq & \frac{k}{2 n (n-1)} \sum_{i=1}^c  (n-\ell_i) \ell_i(1-k/n) \left( - \frac{1}{\ell_i(1-k/n)} + \frac{1}{2 \ell_i^2(1-k/n)^2} \right) 
\nonumber \\
& = &  \sum_{i=1}^c  - \frac{k(n-\ell_i)}{2 n(n-1)} + \frac{k (n-\ell_i)}{4n(n-1) \ell_i(1-k/n)}    \label{eq:v2summand} \\
& \leq & - \frac{ k (c-1)}{2(n-1)} + \frac{k}{4(n-k)(n-1)} \sum_{i=1}^c \left( \frac{n}{\ell_i} -1 \right),  \label{eq:v2final}
\end{eqnarray}
where we used the fact that $\sum_{i=1}^c (n-\ell_i) = n(c-1)$.

Since $k\leq n/2$ by assumption, we know from~\eqref{eq:m3val} 
that $M_3(S_r)$ has the same sign as $(n- 2 \ell_r)$. 
And since, as described in~\eqref{eq:swap}, we can assume 
that $\ell_1 \leq \ell_2 \leq \cdots \leq \ell_c$, we know 
that $\ell_r \leq n/2$ for all $r \leq c-1$,
and hence $M_3(S_r) \geq 0$. This means that, 
since $ \psi''(\ell_i(1-k/n)+1)$ is increasing in $\ell_i$,  
we can bound each 
$\psi''(\ell_i(1-k/n)+1) \leq \psi''(\ell_c(1-k/n)+1) $ so that:
\begin{equation}
\frac{1}{6} \sum_{i=1}^c \psi''(\ell_i(1-k/n)+1) M_3(S_i) 
\leq   \frac{1}{6} \sum_{i=1}^c \psi''(\ell_c(1-k/n)+1)  M_3(S_i). 
\label{eq:v3final}
\end{equation}
Finally, writing $f(\ell) = \ell(n-\ell)(n-2 \ell)$, we 
have that $f(\ell)+f(m-\ell)$ is concave and minimized 
at $\ell =1$ and $\ell =m-1$. Therefore,
the sum $\sum_{r=1}^c f(\ell_r)$ is minimized overall
when $\ell_1 = \cdots = \ell_{c-1} =1$ and $\ell_c = n-(c-1)$, giving the 
value $(c-1)(c-2) (3n-2c) \geq 0$. This means 
that $\sum_{i=1}^c M_3(S_i) \geq 0$, so~\eqref{eq:v3final} is negative, 
and we can simply use~\eqref{eq:v2final} as an upper bound.
\end{proof}

%

\subsection{Proofs for small \texorpdfstring{$\ell$}{}}
\label{sec:appsmallL}

Recall the form of the Newton series expansion;
see, for example~\cite[Eq.~(8), p.~59]{milne:book}:

\begin{lemma} \label{lem:newton}
Consider a function $f:(a,\infty)\to{\mathbb R}$ for some $a>0$.
Suppose that, for some positive integers $k$, $m$ and $y$, 
the $(k+1)$st derivative $f^{(k+1)}(z)$ of $f$ 
is negative for all $m \leq z \leq y$. 
Then, for any integer $x$ satisfying 
$m \leq x \leq y$, there exists some 
$\xi = \xi(x) \in (m,y)$ such that:
\begin{eqnarray*} 
 f(x) & =  & \sum_{r=0}^k \frac{\Delta^r f(m)}{r!} (x-m)_r 
	+  \frac{ f^{(k+1)}(\xi)}{(k+1)!} 
 (x-m)_{k+1} \\
& \leq &   \sum_{r=0}^k \frac{\Delta^r f(m)}{r!} (x-m)_r, 
\end{eqnarray*}
where as before $(x)_k$ represents the falling factorial and 
$\Delta^r$ is the $r$th compounded finite difference 
$\Delta$ where, as usual $\Delta f(x) = f(x+1) - f(x)$.
\end{lemma}


\begin{lemma} \label{lem:smallL2} For $\ell_i \geq 3$ we 
have the upper bound:
\begin{align}
- \ep U(\ell_i,S_i) 
 \leq &\;
\frac{(k)_2}{2(n)_2} \left( -\ell_i(\ell_i-1) \log 
\left( \frac{\ell_i}{\ell_i-1} \right) \right)  \nonumber \\
& \; + \frac{(k)_3}{6(n)_3} \left( - \ell_i(\ell_i-1)(\ell_i-2)  \log \left( \frac{(\ell_i-1)^2}{\ell_i(\ell_i-2)} \right) \right).  \label{eq:newton1}
\end{align}
\end{lemma}
\begin{proof} We write $-U(\ell,s) = - s \log \ell - \log \Gamma(\ell-s+1) + \log \Gamma(\ell)$,
and use Lemma~\ref{lem:newton} with $f(s) = -\log \Gamma(\ell-s+1)$ for fixed $\ell$.

Since $\Delta f(0) = f(1) -f(0) = \log \Gamma(\ell+1) 
- \log \Gamma(\ell) = \log \ell$, the Newton series expansion
of Lemma~\ref{lem:newton} with $m=0$ and $k=2$ gives,
\begin{eqnarray*}
-U(\ell,s) & = & f(s) - f(0) - s \Delta f(0) \\
& \leq & - \frac{s(s-1)}{2} \log \left( \frac{\ell}{\ell-1} \right)  -
\frac{s(s-1)(s-2)}{6} \log \left( \frac{(\ell-1)^2}{\ell(\ell-2)} \right), 
\end{eqnarray*}
where we used the fact that $f^{(3)}(\xi) \leq 0$. 
The result follows on taking expectations, using 
the values of the factorial moments from Lemma~\ref{lem:fallmom}.
\end{proof}

\begin{remark} \label{rem:consist}
Note that~\eqref{eq:newton1} remains valid for $\ell_i = 1,2$ as well, 
as long as we interpret the two bracketed terms 
in the bound appropriately:
For $\ell_i=1$ we assume that both brackets are zero, which is consistent 
with the fact that $U(1,0) = U(1,1) = 0$, so the expectation is zero. 
For $\ell_i =2$, the first bracket equals $-2 \log 2$ and again we assume 
the second bracket is zero. This is consistent with the fact that 
$U(2,0) = U(2,1) = 0$ and $U(2,2) = \log 2$, coupled with the fact 
that $H(n,k,2;2) = k(k-1)/(n(n-1)).$
\end{remark}

\begin{proof}[Proof of Proposition~\ref{prop:alt}]
The key is to deal with the two $\ep U(\ell_i,S_i)$ terms in~\eqref{eq:relentsumform} using two separate arguments, and to
think of $D \left( n,k, \vc{\ell} \right)$  as:
\begin{eqnarray} 
\lefteqn{ \left[ U(n,k) - U(n-\ell, (n-\ell) k/n) \right]  } \nonumber \\
& & + 
\left[ - \ep U(\ell, S_1) \right]
+ \left[ U(n-\ell, (n-\ell) k/n) - \ep U(n-\ell,S_2) \right].
\label{eq:smallL}
\end{eqnarray}
We treat the three terms of~\eqref{eq:smallL} separately.
\begin{enumerate}
\item
First, using the bounds of Proposition~\ref{prop:Ubd},
we obtain that the first term of~\eqref{eq:smallL} is bounded above as,
\begin{eqnarray}
\lefteqn{ U(n,k) - U(n-\ell, (n-\ell) k/n) } \nonumber \\
& \leq & A(n,k) - U(n-\ell, (n-\ell) k/n) + \varepsilon(n-\ell, (n-\ell) k/n) 
\nonumber \\
& = &
\ell( (1-k/n) \log(1-k/n) + k/n) - \frac{k \ell}{12 n(n-\ell)(n-k)} 
\nonumber \\ 
& & + \varepsilon(n- \ell,(n-\ell)k /n) \nonumber \\
& \leq & \ell( (1-k/n) \log(1-k/n) + k/n),
\label{eq:smallL1}
\end{eqnarray}
where the last inequality follows by using the facts 
that $n-\ell \geq n/2$ and $1-k/n \geq 1/2$ to get:
\begin{eqnarray*} \varepsilon(n - \ell, (n-\ell)k/n) & \leq &
\frac{1}{360 (n-\ell)^3 (1-k/n)^3} \nonumber \\
& \leq & \frac{2}{45 n^2 (n-\ell) (1-k/n)} \leq \frac{k \ell}{12 n (n-\ell) (n-k)}.
\end{eqnarray*} 

\item For the second term of~\eqref{eq:smallL},
we apply Lemma~\ref{lem:smallL2} with the same
making the same conventions as in Remark~\ref{rem:consist},
to obtain the upper bound:
\begin{equation}
\frac{(k)_2}{2(n)_2} \left( -\ell(\ell-1) \log \left( \frac{\ell}{\ell-1} \right) \right) 
+ \frac{(k)_3}{6(n)_3} \left( - \ell(\ell-1)(\ell-2)  \log \left( \frac{(\ell-1)^2}{\ell(\ell-2)} \right) \right).  \label{eq:smallL2}
\end{equation}

\item Using Lemma~\ref{lem:sumform2} and taking $\ell_i = n-\ell$ in 
each of the summands in~\eqref{eq:v2summand},
shows that the third term of~\eqref{eq:smallL} is bounded above by:
\begin{eqnarray}
\lefteqn{ - \frac{k \ell}{2 n(n-1)} + \frac{k \ell}{4(n-1) (n-\ell)(n-k)} + \frac{M_3(S_2)}{6} 
 \psi''((n-\ell)(1-k/n)+1) } \label{eq:interm} \\
& \leq & - \frac{k \ell}{2 n(n-1)} + \frac{k \ell}{4(n-1) (n-\ell)(n-k)} + \frac{k \ell}{2n(n-1)(n-2)} \nonumber \\
& \leq & - \frac{k \ell}{2 n(n-1)} + \frac{k \ell}{(n-1) (n-\ell)(n-k)}. \label{eq:smallL3}
 \end{eqnarray}
Here we simplified~(\ref{eq:interm})
by using the fact that, by~\cite[Theorem~9]{alzer:97},
we have,
$$- \psi''(x+1) \leq \frac{1}{(1+x)^2} + \frac{2}{(1+x)^3}
 \leq \frac{3}{x^3},$$
so that,
$$M_3(S_2) (\psi''((n-\ell)(1-k/n)+1) 
\leq\frac{k(n-2k) \ell(n-2\ell)}{2n(n-1)(n-2)(n-k)(n-\ell)}
\leq\frac{k\ell}{2n(n-1)(n-2)}.$$
\end{enumerate}
The result follows on adding~\eqref{eq:smallL1},~\eqref{eq:smallL2} and~\eqref{eq:smallL3}.
\end{proof}

\begin{remark} \label{rem:topsoe}
The last two terms in the bound of Proposition~\ref{prop:alt}
can be further simplified as follows.
Recall, e.g.\ from~\cite[Eq.~(3)]{topsoe:07},
that, for all $x \geq 0$:
\begin{equation} \label{eq:topsoe}
\frac{2 x}{2+x} \leq \log (1+x) \leq \frac{x(2+x)}{2(1+x)}.
\end{equation}
Taking $x=1/(\ell-1)$ in~\eqref{eq:topsoe} and rearranging gives that:
$$ 0  \leq   - \ell (\ell-1) \log \left( \frac{\ell}{\ell-1} \right) + \ell - 1/2 \leq \frac{1}{2(2 \ell -1)}.$$
 Similarly, taking $x = 1/(\ell(\ell-2))$  in~\eqref{eq:topsoe} and rearranging gives that:
$$ 0 \leq \frac{1}{2(\ell-1)} \leq  - \ell (\ell-1)(\ell-2) \log \left( \frac{(\ell-1)^2}{\ell(\ell-2)} \right) + \ell - 1
\leq \frac{\ell-1}{2 \ell^2 - 4 \ell + 1}.$$
\end{remark}

\newpage

\bibliography{ik}

\end{document}